\documentclass[11pt,a4paper,english]{article}
\usepackage{amssymb}
\usepackage{amsfonts}
\usepackage[all]{xy}
\usepackage{amsfonts,amscd,amssymb,amsthm}
\usepackage{amsfonts,amssymb,amsmath,amsxtra,amscd,amsthm,eucal,graphicx,graphics}
\usepackage{tikz}
\usepackage{enumerate}
 \usepackage[T1]{fontenc}
\usepackage[utf8]{inputenc}
\usepackage{babel}
\usepackage{listings}
\newcommand{\overbar}[1]{\mkern 1.5mu\overline{\mkern-1.5mu#1\mkern-1.5mu}\mkern 1.5mu}

\theoremstyle{definition}
\newtheorem{theorem}{Theorem}[section]
\newtheorem{remark}[theorem]{Remark}
\newtheorem{proposition}[theorem]{Proposition}
\newtheorem{lemma}[theorem]{Lemma}
\newtheorem{corollary}[theorem]{Corollary}
\newtheorem{definition}[theorem]{Definition}
\newtheorem{conjecture}[theorem]{Conjecture}

\newtheorem{question}[theorem]{Question}
\newtheorem{example}[theorem]{Example}

 \title{Restraints Permitting the Largest Number of Colourings}
    \author{Jason Brown \and Aysel Erey\footnote{Corresponding author, e-mail: aysel.erey@gmail.com}}
    \date{Department of Mathematics and Statistics\\ Dalhousie University \\ Halifax, Nova Scotia, Canada B3H 3J5 \\[\baselineskip] \today }

\begin{document}

\maketitle

\begin{abstract}
A \textit{restraint} $r$ on $G$ is a function which assigns each vertex $v$ of $G$ a finite set of forbidden colours $r(v)$. A proper colouring $c$ of $G$ is said to be \textit{permitted by the restraint r} if $c(v)\notin r(v)$ for every vertex $v$ of $G$.  A restraint $r$ on a graph $G$ with $n$ vertices  is called a \textit{$k$-restraint} if $|r(v)|=k$ and $r(v) \subseteq \{1,2,\dots ,kn\}$ for every vertex $v$ of $G$. In this article we discuss the following problem: among all $k$-restraints $r$ on $G$, which restraints permit the largest number of $x$-colourings for all large enough $x$? We determine such extremal restraints for all bipartite graphs.
\end{abstract}


\section{Introduction}

In a number of applications of graph colourings, constraints on the colour sets naturally play a role. For example, when one sequentially colours the vertices of a graph under a variety of algorithms, lists of forbidden colours dynamically grow at each vertex as neighbours are coloured. In scheduling and timetable problems, individual preferences may constrain the allowable colours at each vertex (cf.\ \cite{kubale}). There is the well-established and well-studied problem (see, for example, \cite{alon}, \cite{chartrand}, Section 9.2 and \cite{tuza}) of {\em list colourings}, where one has available at each vertex $v$ a list $L(v)$ of possible colours, which is equivalent to the remaining colours being forbidden at the node. 

In all these applications, for each vertex $v$ we have a finite list of {\em forbidden} colours $r(v) \subset {\mathbb N}$, and we call the function $r$ a {\em restraint} on the graph $G$; the goal is to colour the graph subject to the restraint placed on the vertex set. More specifically, a proper $k$-colouring $c$ of $G$ is {\em permitted} by restraint $r$ if $c(v) \not\in r(v)$ for all vertices of $v$ of $G$. The question that is often asked is whether there is a proper colouring that is permitted by a specific restraint. Although the ability to find, for a $k$-chromatic graph $G$ and for all non-constant restraints $r:V(G) \rightarrow \{1,2,\ldots,k\}$, a $k$-colouring permitted by $r$ has been used in the construction of critical graphs (with respect to colourings) \cite{toft} and in the study of some other related concepts \cite{amenable,roberts}. Our aim in this paper is to more fully investigate the {\em number} of colourings permitted by a given restraint.\textsl{}


To begin, we shall need a few definitions.  Let $G$ be a graph on $n$ vertices. A \textit{proper x-colouring} of $G$ is a function $f: V(G)\rightarrow \{1, 2, \dots ,x\}$ such that $f(u)\neq f(v)$ for every $uv\in E(G)$.  We say that $r$ is a {\em $k$-restraint} on $G$ if  $|r(u)| = k$ and $r(u)\subseteq \{1,2,\dots ,kn\}$ for every $u\in V(G)$. If $k = 1$ (that is, we forbid exactly one colour at each vertex) we omit $k$ from the notation and use the word {\em simple} when discussing such restraints. If the vertices of $G$ are ordered as $v_1,v_2\dots v_n$, then we usually write $r$ in the form $[r(v_1),r(v_2)\dots ,r(v_n)]$, and when drawing a graph, we label each vertex with its list of restrained colours.



Given a restraint $r$ on a graph $G$, the {\em restrained chromatic polynomial} of $G$ with respect to $r$, denoted by $\pi_{r}(G,x)$, is defined as the number of $x$-colourings permitted by restraint $r$ \cite{extremalaysel}. Note that this function extends the definition of chromatic polynomial, $\pi(G,x)$ because if $r(v) = \emptyset$ for every vertex $v$, then  $\pi_r(G,x) = \pi(G,x)$. Furthermore, it turns out that $\pi_{r}(G,x)$ is a polynomial function of $x$ when $x$ is large enough \cite{extremalaysel}.\\

Our focus will be  on the following question:

\begin{question} Given a graph $G$ and a natural number $k$, among all  $k$-restraints on $G$ what restraints permit the largest/smallest number of $x$-colourings for all large enough $x$?
\end{question}

Since $\pi_r(G,x)$ is a polynomial function of $x$, it is clear that such extremal restraints always exist for all graphs $G$. Let $R_{\operatorname{max}}(G,k)$ (resp. $R_{\operatorname{min}}(G,k)$) be the set of extremal $k$-restraints on $G$ permitting the largest (resp. smallest) number of colourings for sufficiently large number of colours. \\
In  this article, we  first give a complete answer to the minimization part of this question, by determining  $R_{\operatorname{min}}(G,k)$ for all graphs $G$ (Corollary~\ref{constantrestraintcorollary}). We then turn our attention to the more difficult maximization problem. We give two necessary conditions for a restraint to be in $R_{\operatorname{max}}(G,k)$ for every graph $G$ (Theorem~\ref{neccndtn2}), and we show that these necessary conditions are sufficient to determine $R_{\operatorname{max}}(G,k)$ when $G$ is a bipartite graph (Corollary~\ref{bipartiterestraintcorollary}).


 \section{Preliminaries}
 
Similar to the chromatic polynomial, the restrained chromatic polynomial also satisfies an edge deletion-contraction formula. Recall that $G\cdot uv$ is the graph formed from $G$ by contracting edge $uv$, that is, by identifying the vertices $u$ and $v$ (and taking the underlying simple graph).

\begin{lemma}[Edge Deletion-Contraction Formula]\cite{extremalaysel}\label{restcontrac} Let $r$ be any restraint on G, and $uv\in E(G)$. Suppose that $u$ and $v$ are replaced by $w$ in the contraction $G\cdot  uv$. Then
$$\pi_r(G,x)=\pi_r(G-uv,x)-\pi_{r_{uv}}(G\cdot uv,x)$$
where   \[r_{uv}(a) = \left\{ \begin{array}{ll}
                      r(a) & \mbox{ if  $a\neq w$}\\
                      r(u)\cup r(v)       & \mbox{ if $a=w$ }
                       \end{array} \right. \]
 \textit{for each } $a\in V(G\cdot uv).$
\end{lemma}

Given a restraint function $r$ on a graph $G$, let $M_{G,r}$ be the maximum value in $\displaystyle \bigcup_{v \in V(G)} r(v)$ if the set is nonempty and $0$ otherwise. By using Lemma~\ref{restcontrac}, it is easy to see that the following holds.

\begin{theorem}\cite{extremalaysel}
Let $G$ be a graph of order $n$ and $r$ be a restraint on $G$.  Then for all $x\geq M_{G,r}$, the function $\pi_{r}(G,x)$ is a monic polynomial of degree $n$ with integer coefficients that  alternate in sign.
\end{theorem}

Let $A=[x_1,\dots , x_n]$ be a sequence of variables. Then recall that for $i\in \{0,\dots ,n\}$, the $i^{th}$ elementary symmetric function on $A$ is equal to
$$S_i(A)=\sum_{1\leq k_1<\dots <k_i\leq n}x_{k_1}\dots x_{k_i}.$$

\begin{proposition}\label{emptygraphrest}Let $r$ be a restraint function on the empty graph $G=\overbar{K_n}$. Then for all $x\geq M_{G,r}$,
$$ \pi_{r}(G, x)=\prod_{v \in V(G)} {(x - |r(v)|)}= \sum_{i=0}^n (-1)^{n-i}S_i(A)x^{n-i}$$
where $ A=[~|r(v)|~:~v\in V(G)~]$.
\end{proposition}

We now define when two restraints are essentially the same, with respect to the number of permitted proper colourings.

\begin{definition} Let $r$ and $r'$ be two  restraints on $G$. Set 
$\displaystyle{r(G) = \bigcup_{u\in V(G)} r(u)}$ and $\displaystyle{r'(G) = \bigcup_{u\in V(G)} r'(u)}$.
We say that $r$ and $r'$ are \textit{equivalent} restraints, denoted by $r\simeq r'$, if there exists a graph automorphism $\phi$ of $G$ and a bijective function $f:r(G) \mapsto r'(G)$ such that $$f(r(u))=r'(\phi(u))$$ for every vertex $u$ of $G$. If $r$ and $r'$ are not equivalent, then we call them nonequivalent restraints and write $r\not\simeq r'$.
\end{definition}

\begin{example}
Let $G=P_3$ and $v_1,v_2,v_3\in V(G)$ such that $v_iv_{i+1}\in E(G)$. Consider the restraints $r_1=[\{1\},\{2\},\{3\}]$, $r_2=[\{2\},\{1\},\{4\}]$, $r_3=[\{1\},\{1\},\{2\}]$ and $r_4=[\{3\},\{2\},\{2\}]$ (see Figure~\ref{equivrest}). Then $r_1\simeq r_2$, $r_3\simeq r_4$ and $r_1\not\simeq r_3$.
\end{example}

\begin{figure}[htbp]
\begin{center}
\includegraphics[scale=0.3]{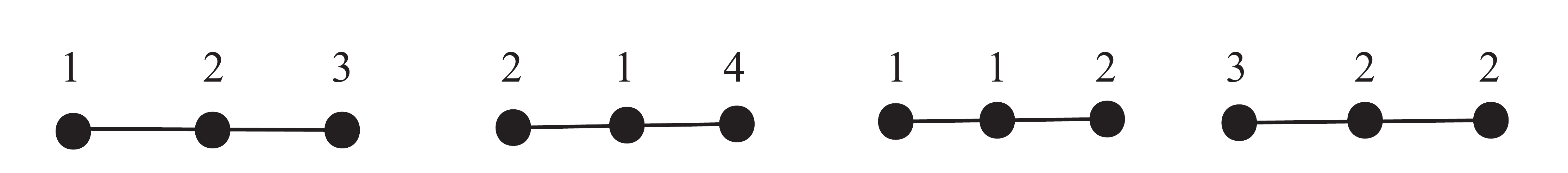}
\end{center}
\caption{Some restraints on $P_3$.}
\label{equivrest}
\end{figure}

It is clear that if $r$ and $r'$ are two equivalent restraints, then $\pi_r(G,x)=\pi_{r'}(G,x)$ for all $x$ sufficiently large. 


\begin{example}
Consider the cycle $C_{3}$. There are essentially three nonequivalent  simple restraints on $C_{3}$, namely, 
\begin{center}
$r_{1}= [\{1\}, \{1\}, \{1\}]$,\\
$r_{2} = [\{1\}, \{2\}, \{1\}]$,\\
$r_{3}= [\{1\}, \{2\}, \{3\}]$.
\end{center}
For $x\geq 3$, the restrained chromatic polynomials with respect to these restraints can be calculated via the Edge Deletion-Contraction Formula as
\begin{eqnarray*}
\pi_{r_1}(C_3,x) & = & (x-1)(x-2)(x-3),\\
\pi_{r_2}(C_3,x) & = & (x-2)(x^2-4x+5), \mbox{ and} \\
\pi_{r_3}(C_3,x) & = & 2(x-2)^2+(x-2)(x-3)+(x-3)^3.
\end{eqnarray*}
where $\pi_{r_1}(C_3,x)<\pi_{r_2}(C_3,x)<\pi_{r_3}(C_3,x)$ holds  for $x>3.$ Hence, $r_3$ permits the largest number of $x$-colourings whereas $r_1$ permits the smallest number of $x$-colourings for large enough $x$.
\end{example}

Given two graphs $G$ and $H$, let $\eta_G(H)$ be the number of subgraphs of $G$ which are isomorphic to $H$, and $i_G(H)$ be the number of induced subgraphs of $G$ which are isomorphic to $H$. In the sequel, we will also need the following result regarding some coefficients of the chromatic polynomial.

\begin{theorem}\label{chromaticcoeff}\cite[pg. 31]{dongbook} If $G$ is a graph of order $n$ and size $m$, then 
$$\pi(G,x)=\sum_{i=1}^{n}(-1)^{n-i}h_i(G)x^i $$
is a polynomial in $x$ such that 
\begin{eqnarray*} 
h_{n-2} & = & {m\choose 2}-\eta_G(C_3) \quad \text{and} \\
h_{n-3} & = & {m \choose 3}-(m-2)\eta_G(K_3)-i_G(C_4)+2\eta_G(K_4).
\end{eqnarray*}
\end{theorem}

Lastly, it is easy to see that if $G_1,G_2,\dots ,G_t$ are connected components of $G$ then
$$\pi_r(G,x)=\prod_{i=1}^t \pi_r(G_i,x).$$
Therefore, $r$ is an extremal restraint for $G$ if and only if $r$ induces an extremal restraint on each connected component. So, we may restrict our attention  only to connected graphs.

\section{Determining $R_{\operatorname{min}}(G,k)$}

A restraint function on a graph $G$ is called \textit{constant k-restraint}, denoted by $r_{c}^k$, if  $r_c^k(u)=\{1,\dots ,k\}$ for every vertex $u$ of $G$. We will show that $r_{c}^k$ permits the smallest number of $x$-colourings for every graph $G$ provided that $x$ is large enough. Observe that  
$$\pi_{r_{c}^k}(G,x) = \pi(G,x-k)$$
 for all $x\geq k$. 

To prove the main results of this section, we will make use of the information about the second and third coefficients of the restrained chromatic polynomial. Hence, first we find combinatorial interpretations for these coefficients. Let $m_G$ denote the number of edges of a graph $G$.

\begin{theorem}\label{restcoefn-1}
Let $r$ be a restraint on a graph $G$, $x\geq M_{G,r}$ and $\displaystyle \pi_r(G,x)=\sum_{i=0}^n (-1)^{n-i}a_i(G,r)x^i$. Then,
$$a_{n-1}(G,r)=m_G+\sum_{u\in V(G)}|r(u)|.$$ 
\end{theorem}

\begin{proof}
We proceed by induction on the number of edges. If $G$ has no edges, then $\displaystyle a_{n-1}(G,r)=\sum_{u\in V(G)}|r(u)|$ by Proposition~\ref{emptygraphrest} and the result clearly holds. Suppose that $G$ has at least one edge, say $e$. Since $G-e$ satisfies the induction hypothesis,
$$\pi_r(G-e,x)=x^n-\left(m_G-1+\sum_{u\in V(G)}|r(u)|\right)~x^{n-1}+\dots$$
holds. Now since $\pi_{r_e}(G\cdot e,x)$ is a monic polynomial of degree $n-1$, the result follows from Lemma~\ref{restcontrac}.
\end{proof}

\begin{theorem}\label{restcoefn-2}
Let $r$ be a restraint on $G$, $x\geq M_{G,r}$ and $\displaystyle \pi_r(G,x)=\sum_{i=0}^n (-1)^{n-i}a_i(G,r)x^i$. Also, let $V(G)=\{u_1,\dots u_n\}$. Then, $a_{n-2}(G,r)$ is equal to
$${m_G \choose{2}}-\eta_G(C_3)+\sum_{i < j}|r(u_i)| ~ |r(u_j)| ~ +m_G \sum_{u_i\in V(G)}|r(u_i)| ~ -\sum_{u_iu_j\in E(G)} |r(u_i)\cap r(u_j)| . $$
\end{theorem}

\begin{proof}
Again we proceed by induction on the number of edges. If $G$ has no edges then $\displaystyle a_{n-2}(G,r)$ is equal to $\sum_{i < j}|r(u_i)| ~ |r(u_j)|$ by Proposition~\ref{emptygraphrest} and the result is clear. Suppose that $G$ has at least one edge, say $e=uv$. Since $G-e$ satisfies the induction hypothesis, the coefficient of $x^{n-2}$ in $\pi_r(G-e,x)$ is equal to
\begin{eqnarray*}
&& {{m_G -1} \choose{2}}~ - ~ \eta_{G-e}(C_3)~ + ~ \sum_{i < j}|r(u_i)| ~ |r(u_j)| ~ + ~(m_G-1)\sum_{u_i\in V(G)}|r(u_i)| \\
&& -\sum_{u_iu_j\in E(G)\setminus \{e\}} |r(u_i)\cap r(u_j)| . 
\end{eqnarray*}
Also, by Theorem~\ref{restcoefn-1}, the coefficient of $x^{n-2}$ in $\pi_r(G\cdot e , x)$ is equal to
$$-m_{G\cdot e}-\sum_{w\in V(G\cdot e)}|r_e(w)|$$
as $G\cdot e$ has $n-1$ vertices.
Observe that $m_{G\cdot e}=m_G-1-|N_G(u)\cap N_G(v)|$ and $|N_G(u)\cap N_G(v)|$ is equal to the number of triangles which contain the edge $uv$. Also, $\eta_{G-e}(C_3)$ is the number of triangles of $G$ which does not contain the edge $uv$. Therefore,
$${{m_G-1}\choose {2}}-\eta_{G-e}(C_3)+m_{G\cdot e} = {{m_G}\choose {2}}- \eta_G(C_3).$$
For a vertex $w$ in $V(G\cdot e)$, by the definition of $r_e$ given in the Edge Deletion-Contraction Formula, $r_e(w)=r(u)\cup r(v)$ if $w$ is obtained by contracting $u$ and $v$, and $r_e(w)=r(w)$ otherwise. Therefore,
\begin{eqnarray*}
\sum_{w\in V(G\cdot e)} |r_e(w)| &=& \sum_{u_i\in V(G)\setminus \{u,v\}}|r(u_i)|~+~|r(u)\cup r(v)|\\
&=& \sum_{u_i\in V(G)\setminus \{u,v\}}|r(u_i)|~+~|r(u)|+|r(v)|-|r(u)\cap r(v)|\\
&=&  \sum_{u_i\in V(G)}|r(u_i)|-|r(u)\cap r(v)|\\
\end{eqnarray*}
Thus, $$(m_G-1)\sum_{u_i\in V(G)}|r(u_i)|~ -\sum_{u_iu_j\in E(G)\setminus \{e\}} |r(u_i)\cap r(u_j)| ~+~ \sum_{w\in V(G\cdot e)}|r_e(w)|$$ is equal to 
$$m_G \sum_{u_i\in V(G)}|r(u_i)| - \sum_{u_iu_j\in E(G)} |r(u_i)\cap r(u_j)|. $$
Hence, the result follows from Lemma~\ref{restcontrac}.
\end{proof}

Now we are ready to answer the question of which $k$-restraint permits the smallest number of  colourings, for a large enough number of colours. 

\begin{theorem}\label{constantrestraint}
 Let $G$ be any connected graph. Then, $r\in R_{\operatorname{min}}(G,k)$ if and only if $r\simeq r_c^k$.
\end{theorem}
\begin{proof}
We shall show that for every $k$-restraint $r$ such that $r\not\simeq r_c^k$, $\pi_{r_c^k}(G,x)<\pi_r(G,x)$ for all large enough $x$.
Both $\pi_{r_{c}^k}(G,x)$ and $\pi_r(G,x)$ are monic polynomials. Also, the coefficient of the term $x^{n-1}$ is the same for these polynomials by Theorem~\ref{restcoefn-1}. Therefore, $ \pi_r(G,x) - \pi_{r_{c}^k}(G,x)$ is a polynomial of degree $n-2$. Now, by Theorem~\ref{restcoefn-2}, the  leading coefficient of $ \pi_r(G,x) - \pi_{r_{c}^k}(G,x)$  is equal to
$$k\, m_G -\sum_{uv\in E(G)}|r(u)\cap r(v)|$$
which is clearly strictly positive as $r$ is a $k$-restraint. Thus, the desired inequality is obtained.
 \end{proof}

\begin{corollary}\label{constantrestraintcorollary}
 Let $G$ be any graph. Then, $r\in R_{\operatorname{min}}(G,k)$ if and only if $r$ induces a constant restraint on each connected component of $G$.
\end{corollary}
 
 \begin{remark}
One can give an alternative proof for the fact that constant $k$-restraint is in $R_{\operatorname{min}}(G,k)$ by using some earlier results regarding  list colourings. But first let us summarize some related work. Kostochka and Sidorenko \cite{kost} showed that if a chordal graph $G$ has a list of $l$ available colours at each vertex, then the number of list colourings is at least $\pi(G,l)$ for every natural number $l$. It is known that there exist graphs $G$ (see, for example, Example~$1$ in \cite{donner}) for which the number of list colourings is strictly less than $\pi(G,l)$ for some natural number $l$. On the other hand, in $1992$, Donner \cite{donner} proved in that for any graph $G$, the number of list colourings is at least $\pi(G,l)$ when $l$ is sufficiently large compared to the number of vertices of the graph. Later, in $2009$, Thomassen proved the same result for $l\geq n^{10}$ where $n$ is the order of the graph. Recently, in \cite{wang} the latter result is improved to $l> \frac{m-1}{\ln (1+ \sqrt{2})}$ by Wang et el.

 As we already pointed out, given a  $k$-restraint $r$ on a graph $G$ and a natural number $x\geq kn$, we can consider an $x$-colouring permitted by $r$ as a list colouring $L$ where each vertex $v$ has a list $L(v)=\{1,\dots ,x\}\setminus r(v)$ of $x-k$ available colours. Therefore, we derive that for a  $k$-restraint $r$ on graph $G$, $\pi_r(G,x) \geq \pi(G,x-k)$ for any natural number $x> \frac{m-1}{\ln (1+ \sqrt{2})}+kn$. But since $\pi_{r_{c}^k}(G,x)$ is equal to $\pi(G,x-k)$, it follows that  $\pi_r(G,x)\geq \pi_{r_{c}^k}(G,x)$ for  $x> \frac{m-1}{\ln (1+ \sqrt{2})}+kn$. Thus, this shows that  $r_c^k\in R_{\operatorname{min}}(G,k)$.

\end{remark}

\section{Two necessary conditions for a restraint to be in $R_{\operatorname{max}}(G,k)$}

The $k$-restraints that permit the smallest number of colourings are easy to describe, and  are, in fact, the same for all graphs.
The more difficult question is which $k$-restraints permit the largest number of colourings; even for special families of graphs, it appears difficult, so we will focus on this question. As we shall see, the extremal $k$-restraints differ from graph to graph.

In this section, we are going to present two results (Theorems~\ref{neccndtn1}, \ref{neccndtn2}) which give necessary conditions for a restraint to be in $R_{\operatorname{max}}(G,k)$ for all graphs $G$. The necessary conditions  given in Theorem~\ref{neccndtn1} and Theorem~\ref{neccndtn2} become sufficient to determine $R_{\operatorname{max}}(G,k)$ when $G$ is a complete graph and bipartite graph respectively. \bigskip

A restraint $r$ on a graph $G$ is called a \textit{proper restraint} if $r(u)\cap r(v)=\emptyset$ for every $uv\in E(G)$. We begin with showing that restraints in $ R_{\operatorname{max}}(G,k)$  must be proper restraints.

\begin{theorem}\label{neccndtn1}
 If $r\in R_{\operatorname{max}}(G,k)$ then $r$ is a proper restraint.
\end{theorem}
\begin{proof}
For  $k$-restraints, from Theorem~\ref{restcoefn-1}, the coefficients of $x^n$ and  $x^{n-1}$  of the restrained chromatic polynomial do not depend on the restraint function. So, in order to maximize the number of $x$-colourings for large enough $x$, one needs to maximize the coefficient of $x^{n-2}$. By Theorem~\ref{restcoefn-2}, it is clear that this coefficient is maximized when $|r(u)\cap r(v)|=0$ for every edge $uv$ of the graph.
\end{proof}

Theorem~\ref{neccndtn1} allows us to determine the extremal restraint for complete graphs, as for such graphs there is a unique (up to equivalence) proper $k$-restraint. We deduce that for complete graphs the extremal restraint is the one where no two vertices have a common restrained colour.

\begin{theorem}\label{completeshortproof}
A restraint $r$ is a proper  $k$-restraint on $K_n$ if and only if $r\in R_{\operatorname{max}}(K_n,k)$.
\end{theorem}

\begin{proof}
 If $r^*$ is a proper $k$ restraint on $K_n$ then $r^*(u)\cap r^*(v)=\emptyset$ for every $u,v\in V(K_n)$. Thus, the result follows from Theorem~\ref{neccndtn1}.
\end{proof}

In general Theorem~\ref{neccndtn1} is not sufficient to determine the extremal restraint. However it is very useful to narrow the possibilities for extremal restraints down to a smaller number of restraints. In the next example, we illustrate this on a cycle of length $4$.

\begin{example}\label{exampleC4-1} Let $G=C_4$ with $V(G)=\{u_1,u_2,u_3,u_4\}$ and  $E(G)=\{u_1u_2, u_2u_3, u_3u_4, u_4u_1\}$. Then there are exactly seven nonequivalent simple restraints on $G$ and these restraints are namely

\begin{center}
$r_{1}= [\{1\}, \{1\}, \{1\}, \{1\}]$,\\
$r_{2} = [\{1\}, \{1\}, \{1\}, \{2\}]$,\\
$r_{3}= [\{1\}, \{1\}, \{2\}, \{2\}]$,\\
$r_{4}= [\{1\}, \{2\}, \{1\}, \{2\}]$,\\
$r_{5}= [\{1\}, \{1\}, \{2\}, \{3\}]$,\\
$r_{6}= [\{1\}, \{2\}, \{1\}, \{3\}]$,\\
$r_{7} = [\{1\}, \{2\}, \{3\}, \{4\}]$.
\end{center}
Now, among these seven restraints, there are only three proper restraints and these are namely  $r_4$, $r_6$ and $r_7$. Therefore, by Theorem~\ref{neccndtn1}, the possibilities for nonequivalent restraints in $R_{\operatorname{max}}(G,k)$ reduce to $r_4$, $r_6$ and $r_7$.
\end{example}
\bigskip

We shall need a combinatorial interpretation for the fourth coefficient of the restrained chromatic polynomial. We will make use of this interpretation in the sequel in order to present another necessary condition for a restraint to be in $R_{\operatorname{max}}(G,k)$.

\begin{theorem}\label{restcoefn-3}
Let $x\geq M_{G,r}$ and $\displaystyle \pi_r(G,x)=\sum_{i=0}^n (-1)^{n-i}a_i(G,r)x^i$. Also, let $V(G)=\{u_1,\dots u_n\}$. Then
$$a_{n-3}(G,r)=A_0(G)~+~\sum_{i=1}^8A_i(G,r),$$  
where
\begin{eqnarray*}
A_0(G) & = &  {{m_G}\choose{3}}-(m_G-2)\eta_G(C_3)-i_G(C_4)+2\eta_G(K_4);\\
A_1(G,r) & =  & \sum_{i<j<k}|r(u_i)|~|r(u_j)|~|r(u_k)|;\\
A_2(G,r) & =  & (m_G-1)\sum_{i<j}|r(u_i)|~|r(u_j)|;\\
A_3(G,r) & =  & \sum_{u_iu_j\notin E(G) \atop i<j}  |r(u_i)|~|r(u_j)|;\\
A_4(G,r) & = & - \sum_{u_iu_j\in E(G)}  |r(u_i)\cap r(u_j)|  \sum _{k\notin \{i,j\}} |r(u_k)|; \\
A_5(G,r) & = &  \left(~{m_G \choose 2}~-~\eta_G(C_3)\right)\sum_{1\leq i\leq n}|r(u_i)|;\\
A_6(G,r) & = &  -(m_G-1)\sum_{u_iu_j\in E(G)}|r(u_i)\cap r(u_j)|;\\
A_7(G,r) & = &  A_7'(G,r)+A_7''(G,r) \quad \text{where}\\
A_7'(G,r) & = & \sum_{u_iu_j\in E(G)}|N_G(u_i)\cap N_G(u_j)|~|r(u_i)\cap r(u_j)|,\\
A_7''(G,r) & =  & -\sum_{u_i\in V(G)} \, \sum_{u_j,u_k \in N_G(u_i) \atop j<k}|r(u_j)\cap r(u_k)|;\\
A_8(G,r) & = & A_8'(G,r)+A_8''(G,r) \quad \text{where}\\
A_8'(G,r) & = & \frac{1}{2}\sum_{u_iu_j \in E(G)}~\sum_{k\notin \{i,j\}\atop u_k\in N_G(u_i)\cup N_G(u_j)}|r(u_i)\cap r(u_j)\cap r(u_k)|\\
A_8''(G,r) & = & \frac{1}{6}\sum_{u_iu_j\in E(G)}~ \sum_{u_k \in N_G(u_i)\cap N_G(u_j)} |r(u_i)\cap r(u_j)\cap r(u_k)|.
\end{eqnarray*}

\end{theorem}

\begin{proof}
We proceed by induction on the number of edges. First suppose that $G$ is an empty graph. We know that  $\displaystyle a_{n-3}(G,r)=A_1(G,r)$ by the formula given in Proposition \ref{emptygraphrest}. Also, it is easy to see that $A_i(G,r)=0$ for $i\notin \{1,2,3\}$, $A_2(G,r)=-\sum_{i<j}|r(u_i)|~|r(u_j)|$ and $A_3(G,r)=\sum_{i<j}|r(u_i)|~|r(u_j)|$. So the result holds for empty graphs. Suppose now that $G$ has at least one edge, say $e=u_1u_2$. First, let us define 

\begin{eqnarray*}
B_0(G,e) &=& {m_{G\cdot e}\choose 2}-\eta_{G\cdot e}(C_3);\\
B_1(G,r,e) &=& 0;\\
B_2(G,r,e) &=& \sum_{i<j}|r(u_i)||r(u_j)|;\\
B_3(G,r,e) &=& -|r(u_1)||r(u_2)|;\\
 B_4(G,r,e) &=& -|r(u_1)\cap r(u_2)|\sum_{i\notin \{1,2\}}|r(u_i)|;\\
B_5(G,r,e) &=& (m_G-1-|N_G(u_1)\cap N_G(u_2)|)\sum_{1\leq i\leq n }|r(u_i)|;\\
B_6(G,r,e) &=& -(m_G-1)|r(u_1)\cap r(u_2)|-\sum_{u_iu_j\in E(G-e)}|r(u_i)\cap r(u_j)|;\\
B_7(G,r,e) &=& |N_G(u_1)\cap N_G(u_2)||r(u_1)\cap r(u_2)|-\sum_{i,j\in\{1,2\} \atop i\neq j}~\sum_{u\in N_G(u_i)\setminus N_G[u_j]}|r(u_j)\cap r(u)|;\\
B_8(G,r,e) &=& \sum _{u_i\in N_G(u_1)\cup N_G(u_2) \atop i \notin \{1,2\}}|r(u_1)\cap r(u_2)\cap r(u_i)|.
\end{eqnarray*}

We shall begin by proving that $$ a_{n-3}(G\cdot e , r_e)=B_0(G,e)+\sum_{i=1}^8B_i(G,r,e).$$

Since $G\cdot e$ has $n-1$ vertices, by Theorem~\ref{restcoefn-2}, the coefficient of $x^{n-3}$ in $\pi_{r_e}(G\cdot e,x)$ is equal to 
\begin{eqnarray*}
&& {m_{G\cdot e}\choose{2}}-\eta_{G\cdot e}(C_3)~+\sum_{u\neq v \atop u,v\in V(G\cdot e)}|r_e(u)||r_e(v)|~+~m_{G\cdot e}\sum_{u\in V(G\cdot e)}|r_e(u)|\\
&& -\sum_{uv \in E(G\cdot e)}|r_e(u)\cap r_e(v)|.\\
\end{eqnarray*}

Now, by the definition of the restraint function $r_e$, we have
\begin{eqnarray*}
\sum_{u\neq v \atop u,v\in V(G\cdot e)}|r_e(u)||r_e(v)| &=& \sum_{3\leq i<j\leq n }|r(u_i)||r(u_j)|+\sum_{i \notin \{1,2\}}|r(u_1)\cup r(u_2)| |r(u_i)|\\
&=& \sum_{3\leq i<j\leq n}|r(u_i)||r(u_j)|+\sum_{k\in\{1,2\}}~\sum_{i\notin \{1,2\}}|r(u_k)||r(u_i)|\\
&&-|r(u_1)\cap r(u_2)|\sum_{i\notin \{1,2\}}|r(u_i)|\\
&=&B_2(G,r,e)+B_3(G,r,e)+B_4(G,r,e).
\end{eqnarray*}

Also, since $m_{G\cdot e}=m_G-1-|N_G(u_1)\cap N_G(u_2)|$ we have 
\begin{eqnarray*}
m_{G\cdot e}\sum_{u \in V(G\cdot e)}|r_e(u)|&=&\left(m_G-1-|N_G(u_1)\cap N_G(u_2)|\right)\left( \left(\sum_{1\leq i\leq n } |r(u_i)|\right)-|r(u_1)\cap r(u_2)|\right)\\
&=& B_5(G,r,e)-\left(m_G-1-|N_G(u_1)\cap N_G(u_2)|\right)|r(u_1)\cap r(u_2)|.
\end{eqnarray*}

Lastly,
\begin{eqnarray*}
-\sum_{uv\in E(G\cdot e)}|r_e(u)\cap r_e(v)|&=& -\sum_{u_iu_j\in E(G) \atop i,j\notin \{1,2\}}|r(u_i)\cap r(u_j)|\\
&& -\sum_{u_i\in N_G(u_1)\cup N_G(u_2) \atop i\notin \{1,2\} }|(r(u_1)\cup r(u_2))\cap r(u_i)|\\
&=&  -\sum_{u_iu_j\in E(G) \atop i,j\notin \{1,2\}}|r(u_i)\cap r(u_j)|\\
&& -\sum_{u_i\in N_G(u_1)\cup N_G(u_2) \atop i\notin \{1,2\}}~\sum_{k\in\{1,2\}}|r(u_k)\cap r(u_i)|\\
&&+\sum_{u_i\in N_G(u_1)\cup N_G(u_2) \atop i\notin \{1,2\}}|r(u_1)\cap r(u_2)\cap r(u_i)|\\
&=&-\sum_{u_iu_j\in E(G-e)}|r(u_i)\cap r(u_j)|\\
&&-\sum_{k,l\in\{1,2\} \atop k\neq l}~\sum_{u_i\in N_G(u_k)\setminus N_G[u_l]}|r(u_i)\cap r(u_l)|\\
&&+\sum_{u_i\in N_G(u_1)\cup N_G(u_2) \atop i\notin \{1,2\}}|r(u_1)\cap r(u_2)\cap r(u_i)|.
\end{eqnarray*}

Thus, by combining all these together we obtain that $ a_{n-3}(G\cdot e , r_e)$ is equal to $B_0(G,e)+\sum_{i=1}^8B_i(G,r,e).$\

Finally, by the edge deletion-contraction formula, it suffices to show that
\begin{eqnarray*}
 A_0(G)&=&A_0(G-e)+B_0(G,e) \quad \mbox{and}\\
 A_i(G,r)&=&A_i(G-e,r)+B_i(G,r,e) \quad \mbox{for} \quad  1\leq i \leq 8.
\end{eqnarray*}

\noindent \textbf{Claim 1:} $A_0(G)=A_0(G-e)+B_0(G,e)$.

\noindent \textit{Proof of Claim 1:} Recall that 
\begin{eqnarray*}
A_0(G) & = & {{m_G}\choose{3}}-(m_G-2)\eta_G(C_3)-i_G(C_4)+2\eta_G(K_4),\\ 
A_0(G-e) & = & {{m_{G-e}}\choose{3}}-(m_{G-e}-2)\eta_{G-e}(C_3)-i_{G-e}(C_4)+2\eta_{G-e}(K_4), \mbox{ and } \\ 
B_0(G,e) & = &  {m_{G\cdot e}\choose 2}-\eta_{G\cdot e}(C_3).
\end{eqnarray*}
By Theorem~\ref{chromaticcoeff}, the coefficient of $x^{n-3}$ in the chromatic polynomial $\pi(G,x)$ of $G$ is equal to $-A_0(G)$. Since $G\cdot e$ has $n-1$ vertices, by Theorem~\ref{chromaticcoeff}, the coefficient of $x^{n-3}$ in the chromatic polynomial $\pi(G\cdot e,x)$ of $G\cdot e$ is equal to $B_0(G,e)$. The chromatic polynomial satisfies the edge deletion-contraction formula, $\pi(G,x)=\pi(G-e,x)-\pi(G\cdot e, x)$. Therefore $-A_0(G)=-A_0(G-e)-B_0(G,e)$ and the result follows.\\

\noindent \textbf{Claim 2:} $A_1(G,r)=A_1(G-e,r)+B_1(G,r,e)$.

\noindent \textit{Proof of Claim 2:} Recall that $\displaystyle A_1(G,r) = A_1(G-e,r)= \sum_{i<j<k}|r(u_i)|~|r(u_j)|~|r(u_k)|$ and $B_1(G,r,e)=0$. Since $G$ and $G- e$ have the same vertices, $A_1(G,r)$ is equal to  $A_1(G-e,r)$. Now the result follows since $B_1(G,r,e)=0$.\\

\noindent \textbf{Claim 3:} $A_2(G,r)=A_2(G-e,r)+B_2(G,r,e)$.

\noindent \textit{Proof of Claim 3:} Recall that 
\begin{eqnarray*}A_2(G,r) & =  & (m_G-1)\sum_{i<j}|r(u_i)|~|r(u_j)|,\\
A_2(G-e,r) & =  & (m_{G-e}-1)\sum_{i<j}|r(u_i)|~|r(u_j)|, \mbox{ and }\\
B_2(G,r,e) & =  & \sum_{i<j}|r(u_i)||r(u_j)|.
\end{eqnarray*}

\noindent Now, $A_2(G-e,r)$ is equal to $(m_G-2)\sum_{i<j}|r(u_i)|~|r(u_j)|$ since $G-e$ has $m_G-1$ edges.\\

\noindent \textbf{Claim 4:} $A_3(G,r)=A_3(G-e,r)+B_3(G,r,e)$.

\noindent \textit{Proof of Claim 4:} Recall that 
\begin{eqnarray*}
A_3(G,r) & =  & \sum_{u_iu_j\notin E(G) \atop i<j}  |r(u_i)|~|r(u_j)|,\\
A_3(G-e,r) & =  & \sum_{u_iu_j\notin E(G-e) \atop i<j}  |r(u_i)|~|r(u_j)|, \mbox{ and }\\
B_3(G,r,e) & =  & -|r(u_1)||r(u_2)|.
\end{eqnarray*}

\noindent The result holds because $E(G)=E(G-e)\cup \{e\}$ and the vertices of $e$ are $u_1$ and $u_2$.\\

\noindent \textbf{Claim 5:} $A_4(G,r)=A_4(G-e,r)+B_4(G,r,e)$.

\noindent \textit{Proof of Claim 5:} Recall that
\begin{eqnarray*}
A_4(G,r) & = & - \sum_{u_iu_j\in E(G)}  |r(u_i)\cap r(u_j)|  \sum _{k\notin \{i,j\}} |r(u_k)|,\\
A_4(G-e,r) & = & - \sum_{u_iu_j\in E(G-e)}  |r(u_i)\cap r(u_j)|  \sum _{k\notin \{i,j\}} |r(u_k)|, \mbox{ and }\\ 
B_4(G,r,e)  & =  & -|r(u_1)\cap r(u_2)|\sum_{i\notin \{1,2\}}|r(u_i)|.
\end{eqnarray*}

\noindent Again, as in the previous case, the result holds because $E(G)=E(G-e)\cup \{e\}$ and the vertices of $e$ are $u_1$ and $u_2$.\\

\noindent \textbf{Claim 6:} $A_5(G,r)=A_5(G-e,r)+B_5(G,r,e)$:

\noindent \textit{Proof of Claim 6:} Recall that 
\begin{eqnarray*}
A_5(G,r) & =  & \left(~{m_G \choose 2}~-~\eta_G(C_3)\right)\sum_{1\leq i\leq n}|r(u_i)|,\\
A_5(G-e,r) & =  & \left(~{m_{G-e} \choose 2}~-~\eta_{G-e}(C_3)\right)\sum_{1\leq i\leq n}|r(u_i)|, \mbox{ and }\\
B_5(G,r,e)  & = &  (m_G-1-|N_G(u_1)\cap N_G(u_2)|)\sum_{1\leq i\leq n }|r(u_i)|.
\end{eqnarray*}

The number of triangles in $G$ is equal to $\eta_G(C_3)$. Observe that $\eta_{G-e}(C_3)$ is the number of triangles in $G$ which does not contain the edge $e$ and $|N_G(u_1)\cap N_G(u_2)|$ is the number of triangles in $G$ which contains the edge $e$. Therefore, $\eta_G(C_3)$ is equal to $\eta_{G-e}(C_3)+|N_G(u_1)\cap N_G(u_2)|$. Also, it is easy to check that ${m_G \choose 2}$ is equal to ${m_{G-e} \choose 2}+m_G-1$ as $m_{G-e}$ is equal to $m_G-1$. Hence, the equality is obtained.\\

\noindent \textbf{Claim 7:} $A_6(G,r)=A_6(G-e,r)+B_6(G,r,e)$:

\noindent \textit{Proof of Claim 7:} Recall that 
\begin{eqnarray*}
A_6(G,r)  & =  & -(m_G-1)\sum_{u_iu_j\in E(G)}|r(u_i)\cap r(u_j)|,\\
A_6(G-e,r)  & =  & -(m_{G-e}-1)\sum_{u_iu_j\in E(G-e)}|r(u_i)\cap r(u_j)|, \mbox{ and }\\
B_6(G,r,e)  & =  & -(m_G-1)|r(u_1)\cap r(u_2)|-\sum_{u_iu_j\in E(G-e)}|r(u_i)\cap r(u_j)|
\end{eqnarray*}

\noindent The reason why the equality holds is the same as in the proofs of Claims $4$ and $5$.\\

\noindent \textbf{Claim 8:} $A_7(G,r)=A_7(G-e,r)+B_7(G,r,e)$:

\noindent \textit{Proof of Claim 8:} Recall that $A_7(G,r)$ is equal to

\[\sum_{u_iu_j\in E(G)}|N_G(u_i)\cap N_G(u_j)|~|r(u_i)\cap r(u_j)| -\sum_{u_i\in V(G)} \, \sum_{u_j,u_k \in N_G(u_i) \atop j<k}|r(u_j)\cap r(u_k)|\,\]

\noindent $A_7(G-e,r)$ is equal to

\[ \sum_{u_iu_j\in E(G-e)}|N_{G-e}(u_i)\cap N_{G-e}(u_j)|~|r(u_i)\cap r(u_j)| -\sum_{u_i\in V(G-e)}  \sum_{u_j,u_k \in N_{G-e}(u_i) \atop j<k}|r(u_j)\cap r(u_k)|\] 
and $B_7(G,r,e)$ is equal to

\[ |N_G(u_1)\cap N_G(u_2)||r(u_1)\cap r(u_2)|-\sum_{i,j\in\{1,2\} \atop i\neq j}~\sum_{u\in N_G(u_i)\setminus N_G[u_j]}|r(u_j)\cap r(u)|.\]

\noindent Observe that $N_G(u_i)=N_{G-e}(u_i)$ for $i\notin \{1,2\}$. Also, $N_G(u_1)\setminus N_{G-e}(u_1)=\{u_2\}$ and $N_G(u_2)\setminus N_{G-e}(u_2)=\{u_1\}$. Therefore, $\displaystyle \sum_{u_iu_j \in E(G)}|N_G(u_i)\cap N_G(u_j)||r(u_i)\cap r(u_j)|$ 

is equal to 
\begin{eqnarray*}
&& \sum_{u_iu_j \in E(G-e)}|N_{G-e}(u_i)\cap N_{G-e}(u_j)||r(u_i)\cap r(u_j)| \\
&& +\sum_{u\in N_G(u_1)\cap N_G(u_2)}(|r(u)\cap r(u_1)|+|r(u)\cap r(u_2)|)\\
&&+|N_G(u_1)\cap N_G(u_2)|~|r(u_1)\cap r(u_2)|.\\
\end{eqnarray*}

Moreover, 
\[ \sum_{u_i\in V(G)}~\sum_{u_j,u_k\in N_G(u_i) \atop j<k}|r(u_j)\cap r(u_k)|\] 
is equal to 
\[ \sum_{u_i\in V(G-e)}~\sum_{u_k,u_j\in N_{G-e}(u_i) \atop j<k}|r(u_k)\cap r(u_j)| +\sum_{s,t\in \{1,2\} \atop s\neq t}~ \sum_{u\in N_G(u_s)\setminus \{u_t\}} |r(u)\cap r(u_t)|.\]
Hence, the result follows since the difference of the sums
\[ \sum_{s,t\in \{1,2\} \atop s\neq t}~ \sum_{u\in N_G(u_s)\setminus \{u_t\}} |r(u)\cap r(u_t)| \, -\sum_{u\in N_G(u_1)\cap N_G(u_2)}(|r(u)\cap r(u_1)|+|r(u)\cap r(u_2)|)\]
can be rearranged as \[ \sum_{i,j\in\{1,2\} \atop i\neq j}~\sum_{u\in N_G(u_i)\setminus N_G[u_j]}|r(u_j)\cap r(u)|.\]


\vspace{0.1in}

\noindent \textbf{Claim 9:} $A_8(G,r)=A_8(G-e,r)+B_8(G,r,e)$:\

\noindent \textit{Proof of Claim 9:} Recall that  $A_8(G,r)$ is equal to

\begin{eqnarray*}
&&\frac{1}{2}\sum_{u_iu_j \in E(G)}~\sum_{k\notin \{i,j\}\atop u_k\in N_G(u_i)\cup N_G(u_j)}|r(u_i)\cap r(u_j)\cap r(u_k)|\\
&&+\frac{1}{6}\sum_{u_iu_j\in E(G)}~ \sum_{u_k \in N_G(u_i)\cap N_G(u_j)} |r(u_i)\cap r(u_j)\cap r(u_k)|,
\end{eqnarray*}
$A_8(G-e,r)$ is equal to
\begin{eqnarray*}
&&\frac{1}{2}\sum_{u_iu_j \in E(G-e)}~\sum_{k\notin \{i,j\}\atop u_k\in N_{G-e}(u_i)\cup N_{G-e}(u_j)}|r(u_i)\cap r(u_j)\cap r(u_k)|\\
&&+\frac{1}{6}\sum_{u_iu_j\in E(G-e)}~ \sum_{u_k \in N_{G-e}(u_i)\cap N_{G-e}(u_j)} |r(u_i)\cap r(u_j)\cap r(u_k)|
\end{eqnarray*}
and $B_8(G,r,e)$ is equal to 
\begin{eqnarray*} 
&&\sum _{u_i\in N_G(u_1)\cup N_G(u_2) \atop i \notin \{1,2\}}|r(u_1)\cap r(u_2)\cap r(u_i)|.
\end{eqnarray*}

It suffices to check two equalities. First,
$$ \frac{1}{2}\sum_{u_iu_j \in E(G)}~\sum_{k\notin \{i,j\}\atop u_k\in N_G(u_i)\cup N_G(u_j)}|r(u_i)\cap r(u_j)\cap r(u_k)|$$ is equal to 
\begin{eqnarray*}
&&\frac{1}{2}\sum_{u_iu_j \in E(G-e)}~\sum_{k\notin \{i,j\}\atop u_k\in N_{G-e}(u_i)\cup N_{G-e}(u_j)}|r(u_i)\cap r(u_j)\cap r(u_k)|\\
&&+\sum_{k\notin \{1,2\} \atop u_k\in (N_G(u_1)\cup N_G(u_2))\setminus (N_G(u_1)\cap N_G(u_2))} |r(u_1)\cap r(u_2)\cap r(u_k)|\\
&&+~\frac{1}{2}\sum_{u\in N_G(u_1)\cap N_G(u_2)}|r(u_1)\cap r(u_2)\cap r(u)|.
\end{eqnarray*}
Secondly, 
$$\frac{1}{6}\sum_{u_iu_j\in E(G)}~ \sum_{u_k \in N_G(u_i)\cap N_G(u_j)} |r(u_i)\cap r(u_j)\cap r(u_k)|$$ is equal to
\begin{eqnarray*}
&&\frac{1}{6}\sum_{u_iu_j\in E(G-e)}~ \sum_{u_k \in N_{G-e}(u_i)\cap N_{G-e}(u_j)} |r(u_i)\cap r(u_j)\cap r(u_k)|\\
&&+~\frac{1}{2}\sum_{u\in N_G(u_1)\cap N_G(u_2)}|r(u_1)\cap r(u_2)\cap r(u)|.
\end{eqnarray*}
Therefore, the result is established.
\end{proof}

\begin{theorem}\label{neccndtn2}
Let G be any graph. If $r^* \in R_{\operatorname{max}}(G,k)$ then $r^*$ satisfies both of the following.
\begin{itemize}
\item[(i)] $r^*$ is a proper restraint,
\item[(ii)] $A_7''(G,r^*)=\operatorname{min}\{A_7''(G,r): \, \text{r is a proper $k$-restraint on G}\}$. In other words,
$$\sum_{u\in V(G)} \sum_{v,w\in N_G(u) \atop v\neq w}|r^*(v)\cap r^*(w)|\geq \sum_{u\in V(G)} \sum_{v,w\in N_G(u) \atop v\neq w}|r(v)\cap r(w)|$$
for every proper k-restraint r on G.
\end{itemize}
\end{theorem}
\begin{proof}
By Theorem~\ref{neccndtn1}, we know that $r^*$ is a proper restraint. So we shall prove the statement in $(ii)$. Let $r$ be a proper $k$-restraint on $G$. Note that $a_n(G,r)=a_n(G,r^*)=1$ as the restrained chromatic polynomial is a monic polynomial. By  Theorem~\ref{restcoefn-1}, we have  $a_{n-1}(G,r)=a_{n-1}(G,r^*)$ as $r$ and $r^*$ are $k$-restraints. Also, since $r$ and $r^*$ are proper restraints we have 
$$\sum_{uv\in E(G)}|r(u)\cap r(v)|=\sum_{uv\in E(G)}|r^*(u)\cap r^*(v)|=0.$$ So, $a_{n-2}(G,r)=a_{n-2}(G,r^*)$ by Theorem~\ref{restcoefn-2}. Since $r^*\in R_{\operatorname{max}}(G,k) $ and the coefficient of $x^{n-3}$ of the restrained chromatic polynomial is negative, we must have $a_{n-3}(G,r)\geq a_{n-3}(G,r^*)$. Recall that  $$a_{n-3}(G,r)=A_0(G)+\sum_{i=1}^8 A_i(G,r)$$ where $A_i(G,r)$'s are as in the statement of Theorem~\ref{restcoefn-3}. First note that $A_0(G)$ does not depend on the restraint function. Furthermore, since $r$ and $r^*$ are $k$-restraints, $A_i(G,r)=A_i(G,r^*)$ for $i=1,2,3,5$. Also, since $r$ and $r^*$ are proper restraints, we have $A_i(G,r)=A_i(G,r^*)=0$ for $i=4,6,8$ and $A_{7}'(G,r)=A_7'(G,r^*)=0$. Thus, $0\leq a_{n-3}(G,r)-a_{n-3}(G,r^*)=A_{7}''(G,r)-A_{7}''(G,r^*)$ and the result follows.
\end{proof}

\begin{example} Let us consider again the graph $C_4$. In Example~\ref{exampleC4-1}, we noted that if $r\in R_{\operatorname{max}}(G,k)$ then $r\in \{r_4,r_6,r_7\}$ where $r_{4}= [\{1\}, \{2\}, \{1\}, \{2\}]$, $r_{6}= [\{1\}, \{2\}, \{1\}, \{3\}]$ and $r_{7} = [\{1\}, \{2\}, \{3\}, \{4\}]$. Now, we apply Theorem~\ref{neccndtn2} to determine $R_{\operatorname{max}}(G,k)$.  We calculate
\begin{center}
$A_7''(G,r_4)=-2$,\\
$A_7''(G,r_6)=-1$,\\
$A_7''(G,r_7)=0$.
\end{center}
Thus, we conclude that $r\in R_{\operatorname{max}}(G,k)$ if and only if $r\simeq r_4$.
\end{example}

In the next theorem, in fact, we will show that  the necessary conditions in Theorem~\ref{neccndtn2} become sufficient to determine the extremal restraints for all bipartite graphs.\bigskip

Suppose $G$ is a connected bipartite graph with bipartition $(V_1,V_2)$. Then a  $k$-restraint is called an \textit{alternating restraint}, denoted $r_{alt}$, if $r_{alt}$ is constant on both $V_1$ and $V_2$ individually (that is, $r_{alt}(a)=r_{alt}(a')$ for every $a,a'\in V_i$ for $i=1,2$), and $r_{alt}(u)\cap r_{alt}(v)= \emptyset$ for every $u\in V_1$ and $v\in V_2$.

\begin{theorem}\label{bipartiterestraint}
Let $G$ be a connected bipartite graph. Then,
 $r\in R_{\operatorname{max}}(G,k)$ if and only if $r\simeq r_{alt}$.
\end{theorem}
\begin{proof}
By Theorem~\ref{neccndtn2}, it suffices to show that for any proper $k$-restraint $r$ such that $r\not\simeq r_{alt}$, 
$$\sum_{u\in V(G)} \sum_{v,w\in N_G(u) \atop v\neq w}|r_{alt}(v)\cap r_{alt}(w)|> \sum_{u\in V(G)} \sum_{v,w\in N_G(u) \atop v\neq w}|r(v)\cap r(w)|.$$
Let $r$ be a  proper $k$-restraint such that $r\not\simeq r_{alt}$. Then there exist distinct vertices $u,v,w$ such that $v,w\in N_G(u)$, and $|r(v)\cap r(w)|<k$, as $G$ is a connected graph. Thus, the result follows since $|r(v)\cap r(w)|=k$ for every $u,v,w$ such that $v,w\in N_G(u)$ and $v\neq w$.
\end{proof}

\begin{corollary}\label{bipartiterestraintcorollary}Let $G$ be a  bipartite graph. Then,
 $r\in R_{\operatorname{max}}(G,k)$ if and only if $r$ induces an alternating restraint on each connected component of $G$.
\end{corollary}

\section{Concluding Remarks}

We have seen that the conditions given in Theorem~\ref{neccndtn2} are sufficient to determine $R_{\operatorname{max}}(G,k)$ when $G$ is a bipartite graph. However these conditions are not sufficient in general to determine the extremal restraints. For example, let $G$ be equal to $C_7$. It is easy to check that if $r$ is a proper simple restraint on $G$ then $|A_7''(G,r)|\leq 4$. Furthermore, for a simple proper restraint $r$ on $G$, $|A_7''(G,r)|= 4$ if and only if  $r$ is equivalent to either $r_1=[ \{1\},\{2\},\{1\},\{2\},\{1\},\{2\},\{3\}]$ or $r_2=[\{1\},\{2\},\{1\},\{2\},\{3\},\{1\},\{3\}]$ (see Figure~\ref{C7figure}). Computer aided computations show that 

$$\pi_{r_1}(G,x)={x}^{7}-14\,{x}^{6}+91\,{x}^{5}-353\,{x}^{4}+879\,{x}^{3}-1404\,{x}^{2
}+1333\,x-581
$$

\noindent and 

$$\pi_{r_2}(G,x)={x}^{7}-14\,{x}^{6}+91\,{x}^{5}-353\,{x}^{4}+880\,{x}^{3}-1411\,{x}^{2
}+1352\,x-600.$$

Therefore, $\pi_{r_2}(G,x)>\pi_{r_1}(G,x)$ for all large enough $x$ and $R_{\operatorname{max}}(G,1)$ consists of restraints which are equivalent to $r_2$. Thus, Theorem~\ref{neccndtn2} cannot determine $R_{\operatorname{max}}(G,1)$ when $G$ is equal to $C_7$. So it remains open to determine $R_{\operatorname{max}}(G,k)$ when $G$ is an odd cycle. For $k=1$, we propose the following conjecture.
 
 \begin{conjecture}  Let $C_n$ be an odd cycle with vertex set $\{v_1,v_2,\dots , v_n\}$ and edge set $\{v_1v_2, v_2v_3, \dots , v_{n-1}v_n, v_1v_n\}$. If $r\in R_{\operatorname{max}}(C_n,1)$, then $r\cong r^*$ where $r^*$ is defined by
 
 \[
r^*(v_i) =
\begin{cases}
1 & \text{if } i\in \{1,3,5,\dots , \frac{n-1}{2}\}\\
2 & \text{if } i\in \{2,4,6,\dots , \frac{n-3}{2}\}\cup \{\frac{n+3}{2}, \frac{n+7}{2}, \frac{n+11}{2},\dots , n\}\\
3 & \text{if } i \in \{\frac{n+1}{2}, \frac{n+5}{2}, \frac{n+9}{2}, \dots , n-1\}
\end{cases}.
\]
 
 \end{conjecture}
 
 In \cite{ayselthesis}, a formula for the fifth coefficient of the restrained chromatic polynomial of a graph with girth at least $5$ was given. We believe that this formula can be used to determine $R_{\operatorname{max}}(C_n,k)$ for every $k\geq 1$.
 
 \begin{question} If $C_n$ is  an odd cycle and $k\geq 1$, then what is $R_{\operatorname{max}}(C_n,k)$?
 \end{question}

\begin{figure}[h]
\begin{center}
\includegraphics[scale=0.3]{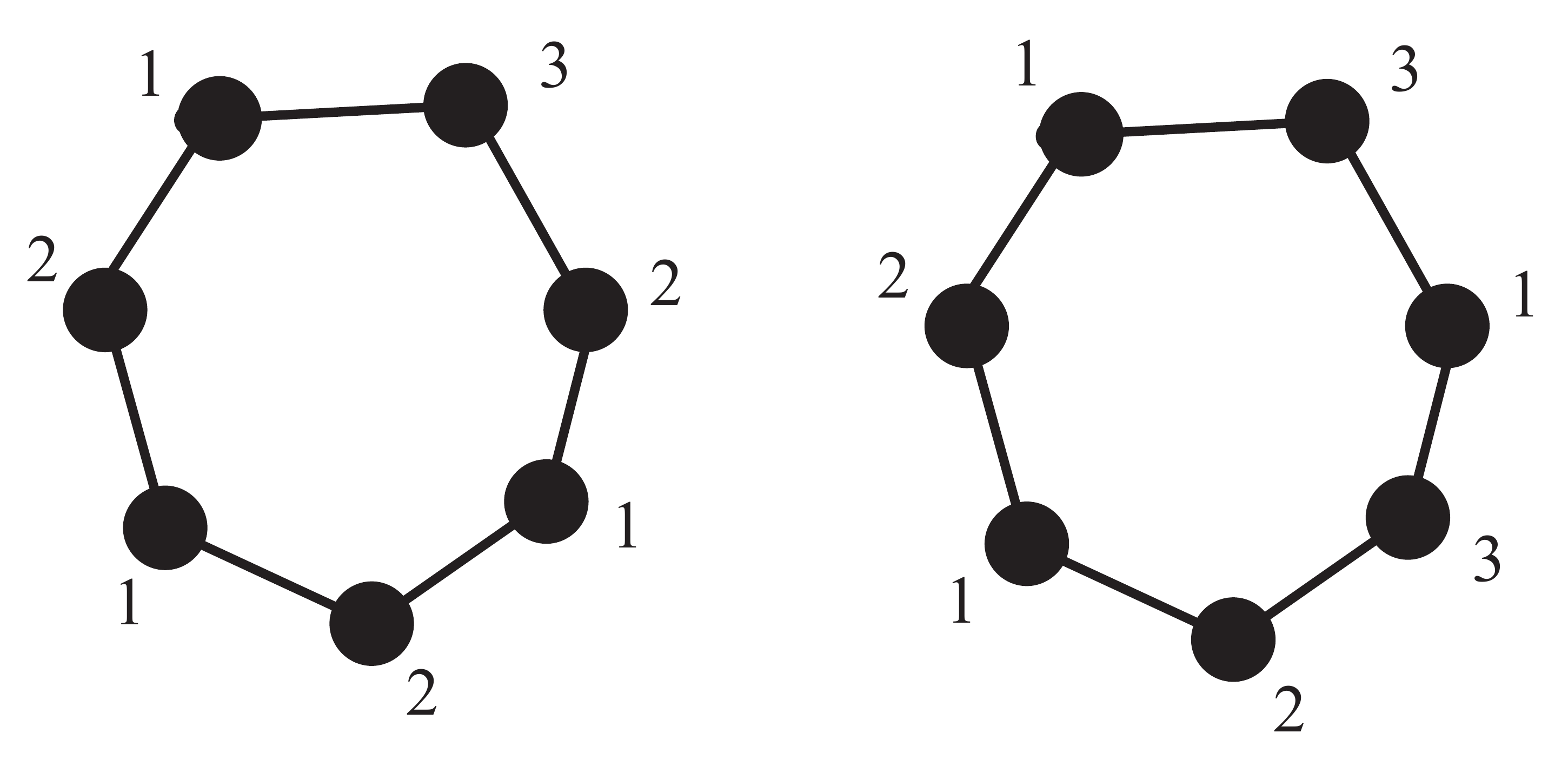}
\end{center}
\caption{Two nonequivalent simple restraints on a cycle graph: $r_1=[\{1\},\{2\},\{1\},\{2\},\{1\},\{2\},\{3\}]$ (left) and  $r_2=[\{1\},\{2\},\{1\},\{2\},\{3\},\{1\},\{3\}]$ (right).}
\label{C7figure}
\end{figure}
\bigskip

In Theorem~\ref{constantrestraint}, we have seen that $R_{\operatorname{min}}(G,k)$ consists of a unique (up to equivalence) $k$-restraint. How about $R_{\operatorname{max}}(G,k)$? We showed that when $G$ is a connected bipartite graph, $R_{\operatorname{max}}(G,k)$ consists of a unique (up to equivalence) $k$-restraint. Does $R_{\operatorname{max}}(G,k)$ consist of a unique (up to equivalence) $k$-restraint for all connected graphs $G$? Or is it possible that $R_{\operatorname{max}}(G,k)$ can contain two nonequivalent restraints for some connected graphs $G$? Note that  there exist graphs for which two nonequivalent restraints permit the same number of colourings. For example, consider the graph $P_4$ with $V(P_4)=\{v_1,v_2,v_3,v_4\}$ and $E(P_4)=\{v_iv_{i+1}\, | \,1\leq i\leq 3 \}$. It is trivial that $r=[\{1\},\{2\},\{2\},\{1\}]$ and $r'=[\{1\},\{2\},\{3\},\{3\}]$ are two nonequivalent restraints on $P_4$. However,
$$\pi_r(P_4,x)=\pi_{r'}(P_4,x)=x^4-7\,x^3+20\,x^2-28\,x+16$$
for all large enough $x$.

Lastly, we proved our results for large enough $x$ but we cannot tell how large $x$ is. So it remains open to determine
 how large $x$ needs to be.

\vskip0.4in
\noindent {\bf \large Acknowledgments:} This research was partially supported by a grant from NSERC. 

\bibliographystyle{elsarticle-num}
\bibliography{<your-bib-database>}

\newpage

\begin{appendix}

\section{Maple program to calculate the restrained chromatic polynomial of a graph}

\begin{lstlisting}
with(GraphTheory):

restchrompoly := proc (G, lst) 
local E, p, i, j, e1, e2, s, e, H, F, p1, p2, V; 
E := Edges(G); 
if E = {} then 
     p := 1; 
     for j to nops(lst) do 
          p := p*(x-nops(lst[j])): 
     od; 
     RETURN(p)
fi; 
V := Vertices(G); 
e := E[1]; 
H := CopyGraph(G); 
DeleteEdge(H, {e}); 
e1 := e[1]; 
e2 := e[2]; 
for i to nops(V) do 
     if V[i] = e1 then 
          p1 := i;
     fi: 
     if V[i] = e2 then 
     p2 := i:
     fi: 
od; 
F := Contract(G, {e[1], e[2]});
s := NULL; 
for i to nops(lst) do 
     if i <> p1 and i <> p2 then 
          s := s, lst[i] :
     fi ; 
     if i = p1 then 
          s := s, lst[p1] union lst[p2]:
     fi:
od; 
s := [s]; 
restchrompoly(H, lst)-restchrompoly(F, s);
end;
\end{lstlisting}

\end{appendix}

\end{document}